\documentclass{amsart}   	

\usepackage{amssymb,amsthm}
\usepackage{amsthm}
\usepackage{mathabx}
\usepackage{graphicx}
\usepackage{enumitem}
\usepackage[utf8]{inputenc}
\usepackage{tikz}
\usepackage{tikz-cd}
\usetikzlibrary{calc}
\usetikzlibrary{decorations.pathmorphing}
\usepackage{hyperref}
\usepackage{adjustbox}
\usepackage{mathtools}
\usepackage[all]{xy}

\tikzset{curve/.style={settings={#1},to path={(\tikztostart)
    .. controls ($(\tikztostart)!\pv{pos}!(\tikztotarget)!\pv{height}!270:(\tikztotarget)$)
    and ($(\tikztostart)!1-\pv{pos}!(\tikztotarget)!\pv{height}!270:(\tikztotarget)$)
    .. (\tikztotarget)\tikztonodes}},
    settings/.code={\tikzset{quiver/.cd,#1}
        \def\pv##1{\pgfkeysvalueof{/tikz/quiver/##1}}},
    quiver/.cd,pos/.initial=0.35,height/.initial=0}

\tikzset{tail reversed/.code={\pgfsetarrowsstart{tikzcd to}}}
\tikzset{2tail/.code={\pgfsetarrowsstart{Implies[reversed]}}}
\tikzset{2tail reversed/.code={\pgfsetarrowsstart{Implies}}}

\tikzset{no body/.style={/tikz/dash pattern=on 0 off 1mm}}

\graphicspath{ {./images/} }

\newtheorem{theorem}{Theorem}[section]
\newtheorem{proposition}[theorem]{Proposition}
\newtheorem{lemma}[theorem]{Lemma}
\newtheorem{corollary}[theorem]{Corollary}

\theoremstyle{definition}
\newtheorem{definition}[theorem]{Definition}
\newtheorem{remark}[theorem]{Remark}
\newtheorem{example}[theorem]{Example} 
 
\newtheorem{conjecture}[theorem]{Conjecture}

\newtheorem{problem}[theorem]{Problem}

\theoremstyle{remark}
\newtheorem*{acknowledgements}{Acknowledgements}

\newcommand{\uxa}{\ensuremath{(\underline{X},\underline{A})}} 
\newcommand{\ux}{\ensuremath{(\underline{X},\underline{\ast})}}

\newcommand{\cxx}{\ensuremath{(\underline{CX},\underline{X})}}
 
\newcommand{\cyy}{\ensuremath{(\underline{CY},\underline{Y})}} 
\newcommand{\clxx}{\ensuremath{(\underline{C\Omega X},\underline{\Omega X})}}

\newcommand{\catK}{\mathrm{cat}(K)}
\newcommand{\cat}{\mathrm{cat}}
\newcommand{\Map}{\mathrm{Map}}
\newcommand{\calD}{\mathcal{D}}
\newcommand{\pxx}{\ensuremath{(\underline{PX},\underline{X})}}
\newcommand{\op}{^{\mathrm{op}}}

\DeclareMathOperator*{\holim}{holim}
\DeclareMathOperator*{\hocolim}{hocolim}

\begin{document}

\title{Polyhedral Coproducts}
\date{}

\author[S.~Amelotte]{Steven Amelotte}
\address{School of Mathematics and Statistics, Carleton University, Ottawa, ON K1S 5B6, Canada}
\email{stevenamelotte@cunet.carleton.ca} 

\author[W.~Hornslien]{William Hornslien}
\address{Institut Fourier - UMR 5582, Université Grenoble-Alpes, CNRS, CS 40700, 38058 Grenoble
Cedex 9, France}
\email{william-proctor.hornslien@univ-grenoble-alpes.fr}

\author[L.~Stanton]{Lewis Stanton}
\address{School of Mathematical Sciences, University of Southampton, University Rd, Southampton, SO17 1BJ, United Kingdom}
\email{lrs1g18@soton.ac.uk}


\begin{abstract}
Dualising the construction of a polyhedral product, we introduce the notion of a \emph{polyhedral coproduct} as a certain homotopy limit over the face poset of a simplicial complex. We begin a study of the basic properties of polyhedral coproducts, surveying the Eckmann--Hilton duals of various familiar examples and properties of polyhedral products. In particular, we show that polyhedral coproducts give a functorial interpolation between the wedge and cartesian product of spaces which differs from the one given by polyhedral products, and we establish a general loop space decomposition for these spaces which is dual to the suspension splitting of a polyhedral product due to Bahri, Bendersky, Cohen and Gitler.
\end{abstract}

\maketitle

\tableofcontents

\section{Introduction}

Polyhedral products are natural subspaces of cartesian products defined as certain colimits over the face poset of a finite simplicial complex $K$. This construction generalises and unifies into a common combinatorial framework many familiar methods of constructing new topological spaces from given ones---for example, products, wedge sums, joins, half-smash products and the fat wedge construction are all special cases. Since their introduction by Bahri, Bendersky, Cohen and Gitler in ~\cite{BBCG1}, the topology of polyhedral products has become a growing topic of investigation within homotopy theory and has made fruitful contact with many other areas of mathematics. Notable examples include toric topology, following Buchstaber and Panov's~\cite{BP} formulation of moment-angle complexes and Davis--Januszkiewicz spaces as polyhedral products; commutative algebra, where polyhedral products give geometric realisations of Stanley--Reisner rings and their Tor algebras; and geometric group theory, where polyhedral products model the classifying spaces of right-angled Artin and Coxeter groups. 
For more on the history and far-reaching applications of polyhedral products, we recommend the excellent survey~\cite{BBC} and references therein.

Motivated by the ubiquity and utility of polyhedral products, the purpose of this paper is to propose a definition for the dual notion of a \emph{polyhedral coproduct} and begin a study of its basic properties. Before describing the main results, we first review the construction of polyhedral products more precisely.

Let $(\underline{X},\underline{A})=\{(X_i,A_i)\}_{i=1}^m$ be an $m$-tuple of pointed CW-pairs, $\mathbf{SCpx}_m$ be the category of simplicial complexes on the vertex set $[m] = \{1,\ldots,m\}$ with morphisms given by simplicial inclusions, and $\mathbf{Top}_*$ be the category of pointed topological spaces. For a simplicial complex $K$, let $\catK$ denote the face poset of $K$, regarded as a small category with objects given by faces $\sigma\in K$ and morphisms given by face inclusions $\tau\subset\sigma$. We denote the initial object of $\catK$ by $\varnothing$, which corresponds to the empty face of $K$. The \emph{polyhedral product} associated to $(\underline{X},\underline{A})$ is the functor
\[ (\underline{X},\underline{A})^{(-)} \colon \mathbf{SCpx}_m \to \mathbf{Top}_* \]
which associates to each simplicial complex $K$ the (homotopy) colimit
\[ (\underline{X},\underline{A})^K=\hocolim_{\catK} \prod_{i=1}^m Y_i(\sigma), \]
where $Y_i\colon \catK \rightarrow \textbf{Top}_*$ is defined for each $i \in [m]$ by 
\[ Y_i(\sigma) = \begin{cases} X_i & \text{if }i \in \sigma \\ A_i & \text{if }i \notin \sigma. \end{cases} \] As has been pointed out in~\cite{KL,NR,WZZ}, for example, the homotopy colimit above agrees up to homotopy with the usual colimit $\bigcup_{\sigma\in K} \prod_{i=1}^m Y_i(\sigma)$ since each $(X_i,A_i)$ is an NDR-pair. In particular, each map in the diagram defining the polyhedral product is a cofibration, and the polyhedral product $(\underline{X},\underline{A})^K$ is a cellular subcomplex of $\prod_{i=1}^m X_i$ for all $K$. In the case that $A_i=*$ for all $i\in [m]$, this subcomplex $(\underline{X},\underline{*})^K$ naturally interpolates between the wedge $\bigvee_{i=1}^m X_i$ (when $K$ consists of $m$ disjoint vertices) and the product $\prod_{i=1}^m X_i$ (when $K=\Delta^{m-1}$ is the simplex on $m$ vertices).

Dualising the definition of a polyhedral product as a homotopy colimit of products, we define a polyhedral coproduct as a homotopy limit of coproducts, as follows.\footnote{Since we work only with homotopy limits, we define polyhedral coproducts with respect to any set of maps $f_i$, rather than insisting on fibrations.}

\begin{definition} \label{def:main}
Let $\underline{f} = (f_1,\ldots,f_m)$ be an $m$-tuple of maps $f_i\colon X_i \rightarrow A_i$ of pointed spaces. Define the \emph{polyhedral coproduct} associated to $\underline{f}$ as the functor 
\[ \underline{f}^{(-)}_{\mathrm{co}} \colon \mathbf{SCpx}_m \to \mathbf{Top}_* \]
which associates to each simplicial complex $K$ with $m$ vertices the homotopy limit 
\[ \underline{f}^K_{\mathrm{co}} = \holim_{\catK\op}  D(\sigma)  \]
of a diagram $D\colon \catK\op \to \textbf{Top}_*$, where $D(\sigma) = \bigvee_{i=1}^m Y_i(\sigma)$ and
\[ Y_i(\sigma) = \begin{cases} X_i & \text{if }i \in \sigma \\ A_i & \text{if }i \notin \sigma. \end{cases} \] 
\end{definition}
 
Note that for a face inclusion $\tau \subset \sigma \in K$, there are maps $Y_i(\sigma) \rightarrow Y_i(\tau)$ defined for each $i \in [m]$ by $f_i$ if $i \in \sigma\backslash\tau$ and by the identity map otherwise, and hence there is an induced map 
\[ \bigvee\limits_{i=1}^m Y_i(\sigma) \to \bigvee\limits_{i=1}^m Y_i(\tau).\]

For a family $(\underline{X}, \underline{A})$ of pairs of spaces, if the maps $f_i\colon X_i \to A_i$ are clear from context, we will sometimes denote $\underline{f}^K_{\mathrm{co}}$ by $(\underline{X}, \underline{A})_{\mathrm{co}}^K$. One example is the case that $A_i=*$ is a point, and $f_i$ is the constant map for all $i\in[m]$. In this case, as we show in Section~\ref{sec:basics}, the polyhedral coproduct $(\underline{X}, \underline{*})_{\mathrm{co}}^K$ naturally interpolates between $\bigvee_{i=1}^m X_i$ (when $K=\Delta^{m-1}$) and $\prod_{i=1}^m X_i$ (when $K$ is $m$ disjoint vertices). 

For polyhedral products, the relationship between the combinatorics of $K$ and the homotopy type of the space $(\underline{X},\underline{*})^K$ interpolating between the $m$-fold wedge and $m$-fold product is made clear after suspending. By~\cite[Theorem~2.15]{BBCG2}, there is a natural homotopy equivalence 
\begin{equation} \label{eq:introBBCG}
\Sigma \ux^K \simeq \bigvee\limits_{\sigma \in K} \Sigma X^{\wedge \sigma},
\end{equation}
where $X^{\wedge \sigma}=X_{i_1}\wedge\cdots\wedge X_{i_k}$ for each face $\sigma=\{i_1,\ldots,i_k\} \in K$. Notice that this generalises the well-known splitting of $\Sigma\big(\prod_{i=1}^m X_i\big)$ when $K=\Delta^{m-1}$, in which case the wedge above is indexed over all subsets of the vertex set $[m]$. For polyhedral coproducts, we dualise the suspension splitting~\eqref{eq:introBBCG} by establishing a loop space decomposition for $(\underline{X}, \underline{*})_{\mathrm{co}}^K$, which similarly generalises a product decomposition due to Porter \cite{P} for $\Omega\big(\bigvee_{i=1}^m X_i\big)$ when $K=\Delta^{m-1}$. 
\begin{theorem}[Theorem~\ref{thm:loopsplituxcase}] \label{thm:introux}
    Let $K$ be a simplicial complex on $[m]$ and let $X_1,\ldots,X_m$ be pointed, simply connected CW-complexes. Then there is a homotopy equivalence \[\Omega \ux^K_{\mathrm{co}} \simeq \prod\limits_{i=1}^{m} \Omega X_i \times \prod\limits_{b\in I} \Omega \Sigma \left(\bigwedge\limits_{\tau \in \mathcal{F}} ((\Omega X)^{\wedge \tau})^{\wedge b(\tau)}\right).\]
\end{theorem}
\noindent As with the splitting~\eqref{eq:introBBCG}, the indexing set $I$ above is defined in terms of the faces of the simplicial complex $K$, and made explicit in Section~\ref{sec:Loopspacesofcoprod}.
The equivalence~\eqref{eq:introBBCG} is a special case of the more general Bahri--Bendersky--Cohen--Gitler splitting (henceforth, BBCG splitting) which identifies the homotopy type of any polyhedral product $(\underline{X},\underline{A})^K$ as a certain wedge after suspending once. In the case that each $X_i$ is contractible, the authors of \cite{BBCG2} obtain their splitting using a lemma regarding homotopy colimits of certain diagrams due to Welker, Ziegler and Živaljević~\cite[Proposition 3.5]{WZZ}. We first dualise the Welker--Ziegler--Živaljević lemma (see Lemma~\ref{lem:initialDiagramLemma}), and then use this to dualise the BBCG splitting to obtain the following result. This is a simplified version of the full statement of Theorem~\ref{thm:loopDomainContractible} where the indexing sets are made explicit in terms of the underlying simplicial complex. 
\begin{theorem}[Theorem~\ref{thm:loopDomainContractible}]
\label{thm:introcontractible}
    Let $K$ be a simplicial complex on $[m]$ and let $f_i\colon 
    X_i \to A_i$ be a map of pointed, simply connected CW-complexes where $X_i$ is contractible for $1 \leq i \leq m$. Then there is a homotopy equivalence
\begin{equation*}
    \Omega \underline{f}^{K} _{\mathrm{co}} \simeq \prod_{b\in I}\Omega \Map_*(\Sigma |K_{I_b}|, \Sigma \Omega A_{1}^{\wedge l_1(b)} \wedge \cdots \wedge \Omega A_{m}^{\wedge l_m(b)}).
\end{equation*}
\end{theorem} 
Both Theorem~\ref{thm:introux} and Theorem~\ref{thm:introcontractible} are special cases of a general loop space decomposition of an arbitrary polyhedral coproduct (see Theorem~\ref{thm:loopSplit}), dual to the suspension splitting of a polyhedral product.

Definition~\ref{def:main} is alternate to Theriault's definition of a \emph{dual polyhedral product}, which was introduced in~\cite{T} and used to identify the Lusternik--Schnirelmann cocategory of a simply connected space $X$ with the homotopy nilpotency of its loop space $\Omega X$. Although the two notions coincide in some special cases (see Remark~\ref{rem:coprodvsdual}), the diagrams defining polyhedral coproducts and dual polyhedral products are very different in general, and our definition is more suitable for dualising the BBCG splitting of $\Sigma(\underline{X},\underline{A})^K$ (see Section~\ref{sec:Loopspacesofcoprod}).

Although we restrict our attention to constructions in $\mathbf{Top}_*$ in this paper, note that polyhedral (co)products could be defined more generally in any model category $\mathcal{C}$, for example, by replacing the category of pointed spaces with $\mathcal{C}$ in the definitions above. Since any (closed) model category has an initial object and a terminal object, the polyhedral products and coproducts of the form $(\underline{X}, \underline{*})^K$ and $(\underline{X}, \underline{*})_{\mathrm{co}}^K$ can be defined in this setting to yield functorial interpolations between the categorical product and coproduct in $\mathcal{C}$.

\begin{acknowledgements}
The authors would like to thank the International Centre for Mathematical Sciences (ICMS), Edinburgh, for support and hospitality during the workshop ``Polyhedral Products: a Path Between Homotopy Theory and Geometric Group Theory", where work on this paper began. We are grateful to the workshop organisers and especially to Martin Bendersky, Mark Grant and Sarah Whitehouse for helpful and encouraging conversations at the beginning of this project. The second author would like to thank Louis Martini for helpful discussions regarding Lemma~\ref{lem:initialDiagramLemma} and the wedge lemma of Welker--Ziegler--Živaljević~\cite{WZZ}. The authors would also like to thank Stephen Theriault for reading a draft of this work, and an anonymous referee for valuable comments that improved the quality of this paper.
\end{acknowledgements}

\section{Basic properties} \label{sec:basics}

\subsection{Basic examples}

We begin by computing some basic examples of polyhedral coproducts, in each case illustrating the Eckmann--Hilton duality between these constructions and their corresponding polyhedral products.
\begin{example}[The $\underline{A}=\underline{*}$ case]\ \label{ex:pointcase} 
    \begin{enumerate}
		\item Let $K$ be $m$ disjoint vertices. In this case, the polyhedral product associated to the $m$-tuple of pairs $(\underline{X},\underline{*})=\{(X_i,*)\}_{i=1}^m$ is the wedge \[ (\underline{X},\underline{*})^K \simeq X_1 \vee \cdots \vee X_m. \] Dually, if $f_i\colon X_i \rightarrow *$ is the constant map for each $i=1,\ldots,m$, then by definition the corresponding polyhedral coproduct is given by
		\[\ux^K_{\mathrm{co}} \simeq X_1 \times \cdots \times X_m.\] 
		\item On the other extreme, let $K=\Delta^{m-1}$. The polyhedral product associated to $(\underline{X},\underline{*})$ in this case is
        \[ (\underline{X},\underline{*})^K \simeq X_1 \times \cdots \times X_m.\]
        Since the diagram defining $\ux^K_{\mathrm{co}}$ has an initial object corresponding to the maximal face of the simplex $\Delta^{m-1}$,
		\[\ux^K_{\mathrm{co}} \simeq X_1 \vee \cdots \vee X_m.\]
		\item Let $K=\partial \Delta^{m-1}$. The polyhedral product $\ux^K$ in this case is precisely the \emph{fat wedge} of the spaces $X_1,\ldots,X_m$, which is defined as \[FW(X_1,\ldots,X_m) = \{(x_1,\ldots,x_m) \mid x_i = * \text{ for at least one } i\}.\] Dual to the fat wedge is the \emph{thin product} of $X_1, \ldots, X_m$, as defined by Hovey in \cite[Definition 1]{Ho}. This construction is realised by the polyhedral coproduct $\ux^K_{\mathrm{co}}$. 
	\end{enumerate}	
\end{example} 

\begin{remark}
\label{rem:coprodvsdual}
    The dual polyhedral product, denoted $\uxa^K_D$, defined by Theriault \cite{T} also models some of the spaces in Example~\ref{ex:pointcase}. In particular, when $K$ is $m$ disjoint points, $\ux^K_D$ is equal to the thin product of $X_1,\ldots,X_m$. When $K = \partial{\Delta^{m-1}}$, $\ux^K_D \simeq X_1 \times \cdots \times X_m$. Outside of these cases, it is not clear whether there is any correspondence between the dual polyhedral product, and the polyhedral coproduct. Theriault also used the dual polyhedral product to give a loop space decomposition of the thin product. An alternate loop space decomposition of the thin product can be recovered in the context of polyhedral coproducts by Theorem~\ref{thm:loopsplituxcase}.
\end{remark}

Just like the polyhedral product $(\underline{X},\underline{*})^K$, the polyhedral coproduct $(\underline{X},\underline{*})_{\mathrm{co}}^K$ interpolates between the categorical product $X_1 \times \cdots \times X_m$ and coproduct $X_1 \vee \cdots \vee X_m$ as $K$ interpolates between a discrete set of vertices and a full simplex. Next, we compute two further examples of $\underline{f}^K_{\mathrm{co}}$ where the $m$-tuple $\underline{f}$ involves maps other than the constant map $X_i\to *$. An important class of polyhedral products (which includes generalised moment-angle complexes $(D^n,S^{n-1})^K$) is given by those associated to CW-pairs $(\underline{CX},\underline{X})=\{(CX_i,X_i)\}_{i=1}^m$ consisting of cones and their bases. The first example below dualises this case by replacing the cofibrations $X_i \hookrightarrow CX_i$ with path space fibrations $PX_i \to X_i$.

\begin{example}[Dual of the join]
	\label{ex:cojoin}
	Let $K=\partial\Delta^1$ be two disjoint vertices so that the only faces of $K$ are $\varnothing$, $\{1\}$ and $\{2\}$, and its face poset is given by $\{1\} \leftarrow \varnothing \rightarrow \{2\}$. In this case the polyhedral product $(\underline{CX},\underline{X})^K$ recovers the join of $X_1$ and $X_2$ as a pushout: 
    \[ (\underline{CX},\underline{X})^K=CX_1\times X_2 \cup_{X_1\times X_2} X_1\times CX_2 \simeq X_1\star X_2. \]
    For $i \in \{1,2\}$, let $f_i \colon PX_i \to X_i$ be the path space fibration over $X_i$. Since each $PX_i$ is contractible, the polyhedral coproduct $\underline{f}_{\mathrm{co}}^K=(\underline{PX},\underline{X})_{\mathrm{co}}^K = \holim(PX_1\vee X_2 \rightarrow X_1\vee X_2 \leftarrow X_1\vee PX_2)$ agrees with the homotopy limit of the middle column of the commutative diagram 
	\begin{equation*}
		\begin{tikzcd}
			* \arrow[d] \arrow[r] & X_1 \arrow[d] \arrow[r] & X_1  \arrow[d]                              \\
			F \arrow[r]     & X_1 \vee X_2 \arrow[r]  & X_1 \times X_2 \\
			* \arrow[u] \arrow[r] & X_2 \arrow[u] \arrow[r] & X_2, \arrow[u]         
		\end{tikzcd}
	\end{equation*} 
    where the vertical maps are inclusions and the rows are homotopy fibrations. The homotopy limit of the right column is contractible, so by taking homotopy limits of the columns we obtain a homotopy equivalence $(\underline{PX},\underline{X})_{\mathrm{co}}^K \simeq \Omega F$. By \cite[p.302]{G}, $F \simeq \Sigma (\Omega X_1 \wedge \Omega X_2)$, and so there is a homotopy equivalence $(\underline{PX},\underline{X})_{\mathrm{co}}^K \simeq \Omega \Sigma (\Omega X_1 \wedge \Omega X_2)$. This space is known as the \textit{cojoin} of $X_1$ and $X_2$.
\end{example}

\begin{example}[Dual of the half-smash]
	\label{ex:cohalfsmash}
Let $K=\partial\Delta^1$ be two disjoint vertices and let $(\underline{X},\underline{A})$ consist of the CW-pairs $\{(X_1,A_1),(X_2,A_2)\}=\{(CX,X),(Y,*)\}$. As in the previous example, the polyhedral product is a pushout $(\underline{X},\underline{A})^K=CX\times * \cup_{X\times *} X\times Y$. Since $CX\times *$ is contractible, this is simply the cofibre of the inclusion $X\times * \hookrightarrow X\times Y$, which by definition is the \emph{half-smash product}
\[ (\underline{X},\underline{A})^K \simeq X\ltimes Y. \]
To dualise this example, let $\underline{f}=(f_1,f_2)$ where $f_1\colon PX\to X$ is the path space fibration and $f_2\colon Y\to *$ is the constant map. Then by definition, the polyhedral coproduct is given by
\begin{align*}
\underline{f}_{\mathrm{co}}^K &= \holim(PX\vee * \rightarrow X\vee * \leftarrow X\vee Y) \\
&\simeq \operatorname{hofib}(X\vee Y \xrightarrow{\pi_X} X),
\end{align*}
the expected Eckmann--Hilton dual of the cofibre $(\underline{X},\underline{A})^K=\operatorname{hocofib}(X\xrightarrow{i_X}X\times Y)$ above. The homotopy fibre of the projection onto a wedge summand can be identified using Mather's Cube Lemma~\cite{Ma}, and we therefore obtain that the dual of the half-smash is given by
\[ \underline{f}_{\mathrm{co}}^K\simeq \Omega X \ltimes Y. \] Moreover, by Mather's Cube Lemma or \cite[Theorem~1.1]{G}, there is a homotopy fibration \[\Sigma \Omega X \wedge \Omega Y \rightarrow \Omega X \ltimes Y \rightarrow Y,\] where the right map is the projection map. The projection has a right homotopy inverse, implying there is a homotopy equivalence \[\Omega(\Omega X \ltimes Y) \simeq \Omega Y \times \Omega (\Sigma \Omega X \wedge \Omega Y).\] This result can be recovered in the context of polyhedral coproducts by Theorem~\ref{thm:loopSplit}.
\end{example}

\subsection{Functorial properties}

The polyhedral product is a bifunctor (see~\cite[Remark~2.3]{BBCG2}). Namely, it defines a functor from the category of ($m$-tuples of) CW-pairs to the category of CW-complexes, and it also defines a functor from the category of simplicial complexes to the category of CW-complexes. In this section, we prove that the polyhedral coproduct enjoys similar functorial properties. First, we show naturality with respect to maps of spaces.

\begin{theorem}
\label{thm:mapinducedbymaps}
    Let $K$ be a simplicial complex on $[m]$. For $1 \leq i \leq m$, let $f_i\colon X_i \rightarrow A_i$ and $f_i'\colon X_i' \rightarrow A_i'$ be maps. If there are maps $g_i\colon X_i \rightarrow X_i'$ and $h_i\colon A_i \rightarrow A_i'$ such that the diagram \begin{equation}\label{eqn:inducedmapfrommaps}\begin{tikzcd}
	{X_i} & {X_i'} \\
	{A_i} & {A_i'}
	\arrow["{h_i}", from=2-1, to=2-2]
	\arrow["{f_i}", from=1-1, to=2-1]
	\arrow["{f_i'}", from=1-2, to=2-2]
	\arrow["{g_i}", from=1-1, to=1-2]
\end{tikzcd}\end{equation} homotopy commutes, then there is an induced map $\underline{f}_{\mathrm{co}}^K \rightarrow \underline{f'}_{\mathrm{co}}^K$.
\end{theorem}
\begin{proof}
    Let $D_K$ and $D'_K$ be the diagrams defining $\underline{f}_{\mathrm{co}}^K$ and $\underline{f'}_{\mathrm{co}}^K$ respectively. For a face $\sigma \in K$, define a map $F_\sigma: D_K(\sigma) \rightarrow D'_K(\sigma)$, defined by \[F_\sigma\colon D_K(\sigma) = \bigvee\limits_{i=1}^m Y_i \xrightarrow{\bigvee\limits_{i=1}^m \phi_i} \bigvee\limits_{i=1}^m Y_i' = D'_K(\sigma),\] where $\phi_i = g_i$ if $i \in \sigma$, and $\phi_i = h_i$ if $i \notin \sigma$. By (\ref{eqn:inducedmapfrommaps}), $F_\sigma$ induces a natural transformation $D_K \rightarrow D'_K$, which in turn induces a map $\underline{f}_{\mathrm{co}}^K \rightarrow \underline{f'}_{\mathrm{co}}^K$.
\end{proof}

The definition of $\underline{f}_{\mathrm{co}}^K$ is also natural with respect to simplicial inclusions.

\begin{theorem}
\label{thm:incinducedmap}
    Let $K$ be a simplicial complex on $[m]$, and let $L$ be a subcomplex of $K$ on $[n]$ with $n \leq m$. Then the simplicial inclusion $L \rightarrow K$ induces a map $\underline{f}_{\mathrm{co}}^K \rightarrow \underline{f}_{\mathrm{co}}^L$.
\end{theorem}

\begin{proof}
   Let $D_K$ and $D_L$ be the diagrams defining $\underline{f}^K_{\mathrm{co}}$ and $\underline{f}_{\mathrm{co}}^L$ respectively. Let $D_L^K\colon \cat(L)\op \to \mathbf{Top}_*$ be the diagram with $D_L^K(\sigma) = \bigvee\limits_{i=1}^m Y_i(\sigma)$, where $Y_i(\sigma) = X_i$ if $i \in \sigma$, and $Y_i(\sigma) = A_i$ if $i \notin \sigma$. By definition of $\underline{f}^K_{\mathrm{co}}$ as a homotopy limit, there are canonical maps $\underline{f}^K_{\mathrm{co}} \to D_L^K(\sigma)$ for all $\sigma \in L$, and so the inclusion $\cat(L)\op \rightarrow \cat(K)\op$ induces a map $\displaystyle \underline{f}_{\mathrm{co}}^K \rightarrow \holim_{\cat(L)\op} D_L^K$.
   
    Now define a natural transformation of diagrams $D_L^K \rightarrow D_L$ by the pinch map \[D^K_L(\sigma) = \bigvee\limits_{i=1}^m Y_i \rightarrow \bigvee\limits_{i=1}^n Y_i = D_L(\sigma).\] This induces a map $\displaystyle \holim_{\cat(L)\op} D_L^K \rightarrow \underline{f}^L_{\mathrm{co}}$. Therefore, the simplicial inclusion induces the composite \[\underline{f}_{\mathrm{co}}^K \rightarrow \holim_{\cat(L)\op} D_L^K \rightarrow \underline{f}_{\mathrm{co}}^L. \qedhere\]
\end{proof}

\begin{remark}
    \label{rem:InclusionDiagram}
    The map $\underline{f}_{\mathrm{co}}^K \rightarrow \underline{f}_{\mathrm{co}}^L$ can be represented as the homotopy limit of a map of diagrams $D_K \to D_L$. For each $\sigma \in L$, we have a pinch map $D_K(\sigma) = \bigvee\limits_{i=1}^m Y_i \rightarrow \bigvee\limits_{i=1}^n Y_i = D_L(\sigma)$. By computing $\holim D_K$, one can see that the maps $\underline{f}_{\mathrm{co}}^K \to D_L(\sigma)$ for $\sigma \in L$ are the maps described in the proof of Theorem~\ref{thm:incinducedmap}.
\end{remark}

\subsection{Retractions}

Let $K$ be a simplicial complex and $L$ a full subcomplex of $K$. For polyhedral products, by \cite[Lemma~2.2.3]{DS}, there is a map $\uxa^K \rightarrow \uxa^L$ which is a left inverse for the map $\uxa^L \rightarrow \uxa^K$. In the case of polyhedral coproducts, there is an analogous statement. 

\begin{theorem}
    Let $K$ be a simplicial complex on $[m]$ and $L$ be a full subcomplex of $K$ on $[n]$, with $n < m$. Then there is a right homotopy inverse for the map $\underline{f}^K_{\mathrm{co}} \rightarrow \underline{f}^L_{\mathrm{co}}$ induced by the simplicial inclusion $L \rightarrow K$.
\end{theorem}
\begin{proof}

 Let $D_K$ and $D_L$ be the diagrams defining $\underline{f}^K_{\mathrm{co}}$ and $\underline{f}_{\mathrm{co}}^L$ respectively. Recall from the proof of Theorem~\ref{thm:incinducedmap} the diagram $D_L^K$ indexed by $\cat(L)\op$, which is defined by $D_L^K(\sigma) = \bigvee\limits_{i=1}^m Y_i(\sigma)$, where $Y_i(\sigma) = X_i$ if $i \in \sigma$, and $Y_i(\sigma) = A_i$ if $i \notin \sigma$. Define a natural transformation $D_L \rightarrow D_L^K$ by the inclusion \[D_L(\sigma) = \bigvee\limits_{i=1}^n Y_i \hookrightarrow \bigvee\limits_{i=1}^m Y_i = D_L^K(\sigma)\] This induces a map $\displaystyle \underline{f}^L_{\mathrm{co}} \xrightarrow{f} \holim_{\cat(L)\op} D_L^K$. Define a functor $F: \cat(K)\op \rightarrow \cat(L)\op$ by sending $\sigma \in K$ to the face $\tau \in L$, where $\tau$ is obtained from $\sigma$ by removing any instances of the vertices $\{n+1,\ldots,m\}$. Since $L$ is a full subcomplex, $F$ is well defined. The functor $F$ induces a map $\displaystyle \holim_{\cat(L)\op} D_L^K \xrightarrow{g} \underline{f}_{\mathrm{co}}^K$. Therefore, we obtain a composite \[\underline{f}^L_{\mathrm{co}} \xrightarrow{f} \holim_{\cat(L)\op} D_L^K \xrightarrow{g} \underline{f}_{\mathrm{co}}^K.\]

Now consider the composite \[\phi:\underline{f}^L \xrightarrow{f} \holim_{\cat(L)\op} D_L^K \xrightarrow{g} \underline{f}_{\mathrm{co}}^K \xrightarrow{h} \holim_{\cat(L)\op} D_L^K \xrightarrow{k} \underline{f}_{\mathrm{co}}^L,\] where the composite $k\circ h\colon \underline{f}_{\mathrm{co}}^K \to \underline{f}_{\mathrm{co}}^L$ is defined as in Theorem~\ref{thm:incinducedmap}. By definition of the functor $F$, the composite $\cat(L)\op \hookrightarrow \cat(K)\op \xrightarrow{F} \cat(L)\op$ is the identity, and so the composite $h \circ g$ is the identity. For a face $\sigma$, the natural transformation inducing the composite $k \circ f$ is the identity on $D_L^K(\sigma)$, and so $k \circ f$ is the identity. Hence, $\phi$ is the identity map, and so the composite $g \circ f$ is a right homotopy inverse for the map induced by $L \rightarrow K$.
\end{proof}

\subsection{Homotopy cofibrations}

For polyhedral products, if $K$ is a simplicial complex on $[m]$, it was shown in \cite[Lemma~2.3.1]{DS} that there exists a homotopy fibration \[\clxx^K \rightarrow \ux^K \rightarrow \prod\limits_{i=1}^m X_i,\] which splits after looping. More generally, it was shown in \cite[Theorem 2.1]{HST} that there is a homotopy fibration \[\cyy^K \rightarrow \uxa^K \rightarrow \prod\limits_{i=1}^m X_i,\] where $Y_i$ is the homotopy fibre of the inclusion $A_i \rightarrow X_i$. Note that the map $\uxa^K \rightarrow \prod\limits_{i=1}^m X_i$ is induced by the inclusion $K \rightarrow \Delta^{m-1}$. Moreover, the second homotopy fibration above also splits after looping, giving a homotopy equivalence \[\Omega \uxa^{K} \simeq \prod\limits_{i=1}^m \Omega X_i \times \Omega \cyy^K.\] This implies that to understand the loop spaces of polyhedral products, and therefore their homotopy groups, it suffices to study polyhedral products of the form $\Omega \cyy^K$. Loop space decompositions of certain polyhedral products of this form have been studied, for example, in \cite{PT,S}. 

For polyhedral coproducts, by considering the map induced by the inclusion $K \rightarrow \Delta^{m-1}$, one might hope there is a homotopy cofibration \[\bigvee\limits_{i=1}^m X_i \rightarrow \ux^K_{\mathrm{co}} \rightarrow (\underline{P \Sigma X},\underline{\Sigma X})_{co}^K,\] or more generally, \[\bigvee\limits_{i=1}^m X_i \rightarrow \uxa^K_{\mathrm{co}} \rightarrow (\underline{PY},\underline{Y})_{co}^K,\] where $Y_i$ is the homotopy cofibre of $f_i\colon X_i \rightarrow A_i$. This would allow us to understand the suspension of polyhedral coproducts, and therefore their homology. However, if $K$ is two disjoint points, by part $(1)$ of Example~\ref{ex:pointcase}, the map $X_1 \vee X_2 \rightarrow \ux^K_{co}$ is the inclusion $X_1 \vee X_2 \rightarrow X_1 \times X_2$. The homotopy cofibre of this map is $X_1 \wedge X_2$, but Example~\ref{ex:cojoin} implies that $(\underline{P \Sigma X},\underline{\Sigma X})_{co}^K \simeq \Omega \Sigma(\Omega \Sigma X_1 \wedge \Omega \Sigma X_2)$. This is reminiscent of how Ganea's theorem~\cite[Theorem~1.1]{G} does not dualise canonically; see~\cite[Remark~3.5]{G}. This gives rise to the following problem.

\begin{problem}
    For certain classes of polyhedral coproducts, determine a decomposition for their suspensions.
\end{problem}

\section{Preliminary Results}
\subsection{Preliminary decompositions}
\label{sec:loopspacedecomp}

To decompose the loop space of a polyhedral coproduct, we will use a result known as the Porter decomposition. Let $K$ be $m$ disjoint points. By \cite[Lemma~2.3.1]{DS}, there is a homotopy fibration 
\begin{equation} \label{DS_fib}
\clxx^K \rightarrow \bigvee\limits_{i=1}^m X_i \xrightarrow{i} \prod\limits_{i=1}^m X_i.
\end{equation}
A result of Porter~\cite[Theorem~1]{P} identifies the homotopy type of $\clxx^K$ in the case that each $X_i$ is simply connected. For a pointed space $X$ and $k \geq 1$, let $X^{\vee k}$ be the $k$-fold wedge of $X$.

\begin{theorem}
\label{htpyfibwedgeintoprod}
    Let $X_1,\ldots,X_m$ be pointed, simply connected CW-complexes, and let $K$ be $m$ disjoint points. There is a homotopy equivalence \[\clxx^K \simeq \bigvee\limits_{k=2}^m \bigvee\limits_{1 \leq i_1 < \cdots < i_k \leq m} (\Sigma \Omega X_{i_1} \wedge \cdots \wedge  \Omega X_{i_k})^{\vee (k-1)}.\] Moreover, this homotopy equivalence is natural for maps $X_i \rightarrow Y_i$.
    \qed
\end{theorem}

There is a special case of the naturality in Theorem~\ref{htpyfibwedgeintoprod} which will be important. Let $n<m$ and let $Y_i = X_i$ for $1 \leq i \leq n$, and let $Y_i = CX_i$ for $n+1 \leq i \leq m$. In this case, we obtain the following. 

\begin{proposition}
\label{naturalporterproj}
    Let $n < m$, and let $X_1,\ldots,X_m$ be pointed, simply connected CW-complexes. There is a homotopy commutative diagram \[\begin{tikzcd}
	{\bigvee\limits_{k=2}^{m} \bigvee\limits_{1 \leq i_1 < \cdots <i_k \leq m} \Sigma (\Omega X_{i_1} \wedge \cdots \wedge \Omega X_{i_k})^{\vee (k-1)}} & {\bigvee\limits_{i=1}^m X_i} & {\prod\limits_{i=1}^m X_i} \\
	{\bigvee\limits_{k=2}^{n} \bigvee\limits_{1 \leq i_1 < \cdots <i_k \leq n} \Sigma (\Omega X_{i_1} \wedge \cdots \wedge \Omega X_{i_k})^{\vee (k-1)}} & {\bigvee\limits_{i=1}^n X_i} & {\prod\limits_{i=1}^n X_i,}
	\arrow[hook, from=1-2, to=1-3]
	\arrow["p", from=1-2, to=2-2]
	\arrow[hook, from=2-2, to=2-3]
	\arrow["\pi", from=1-3, to=2-3]
	\arrow[from=1-1, to=1-2]
	\arrow[from=2-1, to=2-2]
	\arrow["{p'}", from=1-1, to=2-1]
\end{tikzcd}\] where $p$ and $p'$ are pinch maps and $\pi$ is the projection.
\qed
\end{proposition}

After looping the homotopy fibration~\eqref{DS_fib}, there is a natural right homotopy inverse $s$ of the canonical inclusion $i$, given by multiplying the inclusions $X_i \rightarrow \prod_{i=1}^m X_i$. The naturality of $s$ and the homotopy fibration in~Theorem~\ref{htpyfibwedgeintoprod} imply the following.

\begin{theorem}
\label{naturalporterdecomp}
    Let $X_1,\ldots,X_m$ be pointed, simply connected spaces. There is a homotopy equivalence \[\Omega \left(\bigvee\limits_{i=1}^m X_i\right) \simeq \prod\limits_{i=1}^m \Omega X_i \times \Omega \left(\bigvee\limits_{k=2}^m \bigvee\limits_{1 \leq i_1 < \cdots < i_k \leq m} (\Sigma \Omega X_{i_1} \wedge \cdots \wedge \Omega X_{i_k})^{\vee (k-1)}\right).\] Moreover, this homotopy equivalence is natural for maps $X_i \rightarrow Y_i$.
    \qed
\end{theorem}

 For a subset $I = \{i_1,\cdots,i_k\} \subseteq [m]$ and pointed spaces $X_1,\cdots,X_m$, define $X^{\wedge I} = X_{i_1} \wedge \cdots \wedge X_{i_k}$.

\begin{remark}
\label{indexingofPorter}
    Observe that in Theorem~\ref{naturalporterdecomp}, the wedge summand in the right-hand product term can be indexed as \[\bigvee\limits_{I \subseteq [m], |I| \geq 2}  (\Sigma (\Omega X)^{\wedge I})^{\vee (|I|-1)}.\] 
\end{remark}
Now we recall the Hilton--Milnor theorem. Let $L$ be the free (ungraded) Lie algebra over $\mathbb{Z}$ on the elements $x_1,\ldots,x_m$, and let $B$ be a Hall basis of $L$. For a bracket $b \in B$, let $k_i(b)$ be the number of instances of $x_i$ in $b$. For a space $X$ and $k \geq 0$, denote by $X^{\wedge k}$ the $k$-fold smash of $X$. The following is from~\cite[Theorem 4]{Mi}, which is a generalisation of \cite[Theorem A]{Hi}. We will define the $0$-fold smash of $X$ to be omission of the corresponding term, rather than a trivial space.

\begin{theorem}
\label{HiltonMilnor}
    Let $X_1,\ldots,X_m$ be connected topological spaces. Then there is a homotopy equivalence \[\Omega\left( \bigvee\limits_{i=1}^m \Sigma X_i\right) \simeq \prod\limits_{b \in B} \Omega \Sigma(X_{1}^{\wedge k_1(b)} \wedge \cdots \wedge X_m^{\wedge k_m(b)}).\] Moreover, this homotopy equivalence is natural for maps $X_i \rightarrow Y_i$. 
    \qed
\end{theorem}

As in the case of the Porter decomposition, there is a special case which will be important. Let $n<m$ and let $Y_i = X_i$ for $1 \leq i \leq n$, and let $Y_i = CX_i$ for $n+1 \leq i \leq m$. By contracting out the $CX_i$ terms, we obtain the following. 

\begin{corollary}
\label{HiltonMilnorproj}
    Let $n < m$, let $B_n$ be a Hall basis on the free Lie algebra generated by $x_1,\ldots,x_n$, and let $B_m$ be a Hall basis on the free Lie algebra generated by $x_1,\ldots,x_m$. Then the diagram \[\begin{tikzcd}
	{\Omega\Sigma\left(\bigvee\limits_{i=1}^m X_i\right)} & {\prod\limits_{b \in B_m} \Omega \Sigma(X_{1}^{\wedge k_1(b)} \wedge \cdots \wedge X_m^{\wedge k_m(b)})} \\
	{\Omega\Sigma\left(\bigvee\limits_{i=1}^n X_i\right)} & {\prod\limits_{b \in B_n} \Omega \Sigma(X_{1}^{\wedge k_1(b)} \wedge \cdots \wedge X_n^{\wedge k_n(b)}).}
	\arrow["{\simeq }", from=1-1, to=1-2]
	\arrow["\pi", from=1-2, to=2-2]
	\arrow["{\Omega p}", from=1-1, to=2-1]
	\arrow["{\simeq }"', from=2-1, to=2-2]
\end{tikzcd}\] homotopy commutes. \qed

\end{corollary}

\subsection{Preliminary homotopy limit decompositions} 
 In this section, we prove some decompositions of certain homotopy limits indexed by the opposite of the face category of a simplicial complex. The first lemma is the dual statement of the ``Wedge Lemma" from \cite[Proposition~3.5]{WZZ}. 
 
\begin{lemma}
\label{lem:initialDiagramLemma}
    Let $K$ be a simplicial complex. Let $X$ be a space and let $\calD\colon \catK^{\mathrm{op}} \to \mathbf{Top}_* $ be a diagram with $\calD(\varnothing) = X$ and $\calD(\sigma)=*$ for all $\sigma \neq \varnothing$. Then 
    \begin{equation*}
        \holim_{\catK\op} \calD \simeq \Map_*(\Sigma |K|, X).  
    \end{equation*}
\end{lemma}

\begin{proof}
      Let $\catK\op_{>\varnothing}$ denote the subposet category consisting of all $\sigma \in K$ where $\sigma \neq \varnothing$. For a topological space $A$, let $\underline{A} \colon \catK^{\mathrm{op}} \to \mathbf{Top}_*$ be the diagram with constant value $A$. The opposite category $(\catK\op_{>\varnothing})\op$ is the category $\catK_{<\varnothing}$ and its geometric realization coincides with the realization of $K$ as a topological space
     \begin{equation*}
         |K| \simeq \hocolim_{\catK_{<\varnothing}} \underline{*}.
     \end{equation*} 
       Thus there are homotopy equivalences 
     \begin{equation*}
         \Map(|K|,X)\simeq \Map( \hocolim_{\catK_{<\varnothing}} \underline{*},X)\simeq \holim_{\catK\op_{>\varnothing}}\Map(*,X) \simeq \holim_{\catK\op_{>\varnothing}} \underline{X}.
     \end{equation*} 

     Let~$\mathcal{X}\colon \catK^{\mathrm{op}} \to \mathbf{Top}_*$ be the diagram with $\mathcal{X}(\varnothing)=*$ and $\mathcal{X}(\sigma)=X$ for all $\sigma \neq \varnothing$. The diagram~$\calD$ can be written as the (homotopy) pullback 
     \begin{equation}
        \label{eq:diagramPullback}
         \underline{*}\longrightarrow \mathcal{X} \longleftarrow \underline{X},
     \end{equation} where the right-hand map is the constant map to the basepoint for $\sigma = \varnothing$, and the identity on $X$ for $\sigma \neq \varnothing$. The left-hand map is the inclusion of the basepoint for each $\sigma \in K$. 
     
Let $\mathcal{Z}$ be the category with two objects, $1$ and $2$, and a morphism $f\colon 2 \to 1$ in addition to the identity morphisms. Consider the diagram $\mathcal{Y}\colon \catK\op_{>\varnothing} \times \mathcal{Z}\to \mathbf{Top}_*$, with $\mathcal{Y}(\sigma,1) = *$  and $\mathcal{Y}(\sigma,2) = X$ for all $\sigma\neq \varnothing$. Let $F_1\colon \catK\op_{>\varnothing}\times \mathcal{Z} \to \catK\op$ be the functor sending $(\sigma,1) \mapsto \varnothing$ and $(\sigma,2) \mapsto \sigma$. Let $F_2 \colon \catK\op_{>\varnothing} \times \mathcal{Z} \to \cat(\Delta^0)\op$ be the functor sending  $(\sigma,1) \mapsto 1$ and $(\sigma,2) \mapsto 2$ for all $\sigma  \neq \varnothing$. The right Kan extension of $\mathcal{Y}$ along $F_1$ is $\mathcal{X}$ and the right Kan extension of $\mathcal{Y}$ along $F_2$ is  the diagram 
\begin{equation*}
    \Map(|K|,X) \to \Map(|K|,*).
\end{equation*}
Since homotopy limits are preserved by right Kan extensions, we obtain the equivalences
\begin{equation*}
        \holim_{\catK\op} \mathcal{X} \simeq \holim \left(\Map(|K|,X) \to \Map(|K|,*)\right) \simeq \Map(|K|,\holim(X \to *)) \simeq \Map(|K|,X).
    \end{equation*}
    Recall that the diagram $\calD$ was equivalent to diagram (\ref{eq:diagramPullback}). Using that $\displaystyle \holim_{\catK\op} \underline{*}$ is contractible and the previous observations about $\mathcal{X}$ yields the homotopy equivalence
    \begin{equation}
        \label{eq: outerSquare}
        \holim_{\catK\op} \calD\simeq \holim \left(* \longrightarrow \Map(|K|,X) \longleftarrow X\right).
    \end{equation}

    The map $X \rightarrow \Map(|K|,X)$ is a section for the evaluation map $ev_k\colon \Map(|K|,X) \rightarrow X$, where $k \in |K|$. In particular, there is a homotopy fibration diagram \[\begin{tikzcd}
	{\Omega \Map_*(|K|,X)} & {*} & {\Map_*(|K|,X)} \\
	F & X & {\Map(|K|,X)} \\
	& X & X,
	\arrow[from=1-1, to=1-2]
	\arrow["\simeq", from=1-1, to=2-1]
	\arrow[from=1-2, to=1-3]
	\arrow[from=1-2, to=2-2]
	\arrow[from=1-3, to=2-3]
	\arrow[from=2-1, to=2-2]
	\arrow[from=2-2, to=2-3]
	\arrow[equals, from=2-2, to=3-2]
	\arrow["ev_k", from=2-3, to=3-3]
	\arrow[equals, from=3-2, to=3-3]
\end{tikzcd}\] where the top right square is a homotopy pullback and $F$ is the homotopy limit of \eqref{eq: outerSquare}. Hence, there are homotopy equivalences \[\displaystyle \holim_{\catK\op} \calD \simeq F \simeq \Omega \mathrm{Map}_*(|K|,X) \simeq \mathrm{Map}_*(|K|,\Omega X) \simeq \mathrm{Map}_*(\Sigma |K|,X).\qedhere\]
\end{proof}

\begin{lemma}
    \label{lem:diagramContractingArrows}
    Let $K$ be a simplicial complex on $[m]$. Let $I \subseteq [m]$, and let $\calD\colon \catK^{\mathrm{op}} \to \mathbf{Top}_* $ be a diagram. Suppose that all maps induced by $\sigma \subset \tau$, where $\sigma$ is obtained from $\tau$ by removing a single vertex not contained in $I$, are identity maps. Then the homotopy limit of $\calD$ is equivalent to the homotopy limit of a diagram $\calD'\colon \cat(K_{I}) \to \mathbf{Top}_*$, where $\calD'(\sigma_I) = \calD(\sigma)$.
\end{lemma}

\begin{proof}
      The inclusion of $K_I \subset K$ induces a map of face categories $i \colon \cat(K_I)\op \to \cat(K)\op$. The diagram $\calD$ is the right Kan extension of $\calD'$ along $i$. Hence, $\calD$ and $\calD'$ have the same homotopy limit since right Kan extensions preserve homotopy limits. 
\end{proof}

\section{Loop spaces of polyhedral coproducts}
\label{sec:Loopspacesofcoprod}

\subsection{A general loop space decomposition}

In \cite[Definition 2.2]{BBCG2}, for a simplicial complex $K$, a construction known as the \emph{polyhedral smash product} is defined and denoted by $\widehat{\uxa}^K$. By \cite[Theorem 2.10]{BBCG2}, there is a homotopy equivalence \[\Sigma \uxa^K \simeq \bigvee\limits_{I \subseteq [m]} \Sigma \widehat{\uxa}^{K_I}.\] In this subsection, we show a dual statement for polyhedral coproducts. Recall that for spaces $X$ and $Y$ there is a (homotopy) cofibration $X \vee Y \rightarrow X \times Y \rightarrow X \wedge Y$. To dualise this, by \cite[p.302]{G}, there is a homotopy fibration $\Sigma (\Omega X \wedge \Omega Y) \rightarrow X \vee Y \rightarrow X \times Y$. This motivates the following definition.
\begin{definition}    
    The \emph{polyhedral smash coproduct} is defined as the homotopy limit 
    \begin{equation*}
       \underline{\hat{f}}^{K} _{\mathrm{co}} = \holim_{\catK\op}  \Sigma \hat{D}, \quad \text{where } \hat{D} =  \bigwedge_{i=1}^{m} \Omega Y_i \quad \text{and} \quad Y_i(\sigma) = \begin{cases}X_i & \text{if } i \in \sigma, \\ A_i & \text{if } i \not\in \sigma.\end{cases}
    \end{equation*}

    For a set of positive integers $N = \{k_1(N), \ldots, k_m(N) \}$, we define \emph{the weighted polyhedral smash coproduct} as 
    \begin{equation*}
      \underline{\hat{f}}^{K} _{N,\mathrm{co}} = \holim_{\catK\op}  \Sigma \hat{D}^N, \quad \text{where }   \hat{D}^N = \bigwedge_{i=1}^{m} (\Omega Y_i)^{\wedge k_i(N)} \quad \text{and} \quad Y_i(\sigma) = \begin{cases}X_i & \text{if } i \in \sigma, \\ A_i & \text{if } i \not\in \sigma.\end{cases}
    \end{equation*}
\end{definition}

Before stating the result, we set up some notation which will be used throughout the rest of Section~\ref{sec:Loopspacesofcoprod}. For a subset $I = \{i_1,\ldots,i_k\} \subseteq [m]$, let $S_I$ be the set \[\{a_{J,i}\:|\: J \subseteq I, |J| \geq 2, 1 \leq i \leq |J|-1\}.\] Denote by $B_I$ a Hall basis of the free ungraded Lie algebra on the set $S_{I}$. For a bracket $b \in B_I$ and $J \subseteq I$, let $b(J)$ be the sum of the number of instances of $a_{J,i}$ in $b$ for each $1 \leq i \leq |J|-1$. For $1\leq i \leq m$ and a bracket $b \in B_{[m]}$, define \begin{equation*}
    l_i(b) \coloneqq \sum_{I\subseteq [m], i \in I} b(I),
\end{equation*} which counts the number of instances of each vertex $i$ in the faces in $b$.
Let $L_b = (l_1(b), \ldots, l_m(b))$. For any $I \subseteq [m]$ and $b \in B_{[m]}$, define 
\begin{equation*}
    I_b \coloneqq I \cap \{ j \:|\:  1 \leq j \leq m,\; l_j(b) \neq 0\}.
\end{equation*} This set contains the elements of $I$ which appear in the subsets in $b$. To ensure that $\Omega X_i$ is connected in order to apply Theorem~\ref{HiltonMilnor}, we need the hypothesis that each $X_i$ is simply connected.

\begin{theorem}
    \label{thm:loopSplit}
   Let $f_i\colon X_i \rightarrow A_i$ be a map of pointed, simply connected CW-complexes for all $1 \leq i \leq m$. There is a homotopy equivalence 
\begin{equation*}
    \Omega \underline{f}^{K} _{\mathrm{co}} \simeq \prod_{i=1}^m  \Omega X_i \times  \prod_{b \in B_{[m]}} \Omega \underline{\hat{f}}^{K_{I_b}} _{L_b,\mathrm{co}}.
\end{equation*}
\end{theorem}

\begin{proof}

Since loops commute with homotopy limits, $\displaystyle \Omega \underline{f}^{K} _{\mathrm{co}} \simeq \holim_{\catK\op} \Omega D$.  By Theorem~\ref{naturalporterdecomp} and Remark \ref{indexingofPorter}, $\Omega D$ decomposes as a product
\begin{equation*}
   \Omega D \simeq \prod\limits_{i=1}^m \Omega Y_i \times  \Omega \Sigma \left(\bigvee_{I \subseteq  [m], |I| \geq 2} ((\Omega Y)^{\wedge I}) ^{\vee ({|I|-1})}\right).
\end{equation*}
We can apply the Hilton--Milnor theorem (Theorem~\ref{HiltonMilnor}) to the right-hand product term to obtain the equivalence
\begin{equation*}
   \Omega \left( \bigvee_{I \subseteq [m], |I| \geq 2}\Sigma ((\Omega Y)^{\wedge I}) ^{\vee ({|I|-1})} \right)\simeq \prod_{b \in B_{[m]}} \Omega \Sigma \left( \bigwedge_{I \subseteq [m], |I| \geq 2} ((\Omega Y)^{\wedge I})^{\wedge b(I)}\right)  .
\end{equation*}
Note that for any $b \in B_{[m]}$, by definition
\begin{equation*}
    \Sigma \left( \bigwedge_{I \subseteq [m], |I| \geq 2} ((\Omega Y)^{\wedge I})^{\wedge b(I)}\right) = \Sigma \hat{D}^{L_b}.
\end{equation*}
Using the loop space decomposition of $D$ we obtain a new description of $\Omega \underline{f}^{K} _{\mathrm{co}}$:
\begin{align*}
    \Omega \underline{f}^{K} _{\mathrm{co}} & \simeq \holim_{\catK\op} \Omega D 
     \simeq \holim_{\catK\op} \left( \prod_{i=1}^m \Omega Y_i \times \prod_{b \in B_{[m]}} \Omega \Sigma \hat{D}^{L_b} \right)\\
    &\simeq  \prod_{i=1}^m  \left( \holim_{\catK\op} \Omega Y_i \right) \times \prod_{b \in B_{[m]}} \left( \holim_{\catK\op} \Omega \Sigma\hat{D}^{L_b}  \right).
\end{align*}
Fix $i \in [m]$ and consider the term $\displaystyle \holim_{\catK\op} \Omega Y_i$. By Lemma \ref{lem:diagramContractingArrows} in the case $I = \{i\}$, there is an equivalence 
\begin{equation*}
    \holim_{\catK\op} \Omega Y_i \simeq \holim (\Omega X_i \to \Omega A_i) \simeq \Omega X_i.
\end{equation*}
For any $b \in B_{[m]}$ and any pair of simplices $\sigma \subset \tau$ where $\sigma$ is obtained from~$\tau$ by removing a vertex not in $I_b$, the induced maps $\Omega \Sigma \hat{D}^{L_b} (\sigma) \to \Omega \Sigma \hat{D}^{L_b}(\tau)$ are identity maps. Therefore, Lemma~\ref{lem:diagramContractingArrows} implies that
\begin{equation*}
    \holim_{\catK\op} \Omega\Sigma \hat{D}^{L_b} \simeq \holim_{\cat( K_{I_b})\op} \Omega \Sigma\hat{D}^{L_b}  \simeq \Omega \underline{\hat{f}}^{K_{I_b}} _{L_b,\mathrm{co}}.\qedhere
\end{equation*}
\end{proof}

\subsection{Loop space decompositions of \texorpdfstring{$(\underline{X},\underline{\ast})^K_{\mathrm{co}}$}{(X,*)co}}

For polyhedral products of the form $\ux^K$, by \cite[Theorem~2.15]{BBCG2}, there is a homotopy equivalence \[\Sigma \ux^K \simeq \bigvee\limits_{\sigma \in K} \Sigma X^{\wedge \sigma}.\] In this subsection, we prove a dual statement for polyhedral coproducts of the form $\ux^K_{\mathrm{co}}$. Let $\mathcal{F}$ and $\mathcal{M}$ be the set of faces and maximal faces of $K$ on two or more vertices, respectively. The following result could be shown using Theorem~\ref{thm:loopSplit} by showing that certain polyhedral smash coproducts are contractible in this case. However, this would require a technical argument involving choices of vector space bases for free Lie algebras. To avoid these technicalities, and make clearer the connection to Hall bases, we provide a proof using Corollary~\ref{HiltonMilnorproj}.

\begin{theorem}
\label{thm:loopsplituxcase}
    Let $X_1,\ldots,X_m$ be pointed, simply connected CW-complexes. There is a homotopy equivalence \[\Omega \ux^K_{\mathrm{co}} \simeq \prod\limits_{i=1}^{m} \Omega X_i \times \prod\limits_{b \in \bigcup\limits_{\sigma \in \mathcal{M}} B_{\sigma}} \Omega \Sigma \left(\bigwedge\limits_{\tau \in \mathcal{F}} ((\Omega X)^{\wedge \tau})^{\wedge b(\tau)}\right).\]
\end{theorem}
\begin{proof}
    By definition of the polyhedral coproduct, $\displaystyle \ux^K_{\mathrm{co}} = \holim_{\catK\op} D$, where, if $\sigma = \{i_1,\ldots,i_k\}$, $D(\sigma) = \bigvee_{j=1}^k X_{i_j}$, and for each $\sigma' \subset \sigma$, the map $D(\sigma) \rightarrow D(\sigma')$ is given by the pinch map $p:\bigvee_{i \in \sigma} X_i \rightarrow \bigvee_{j \in \sigma'} X_j$. Since looping commutes with homotopy limits, we obtain a homotopy equivalence $\displaystyle \Omega\ux^K_{\mathrm{co}} \simeq \holim_{\catK\op} \Omega D$. By Theorem~\ref{naturalporterdecomp} and Remark~\ref{indexingofPorter}, there is a natural homotopy equivalence 
    \begin{equation}\label{Porteronuxcase}\Omega \left(\bigvee_{j=1}^k X_{i_j}\right) \simeq \prod\limits_{j=1}^k \Omega X_{i_j} \times \Omega \left(\bigvee\limits_{\tau \subseteq \sigma,|\tau| \geq 2} \left(\Sigma(\Omega X)^{\wedge \tau}\right)^{\vee|\tau|-1}\right).\end{equation} 
    Under this equivalence, it follows from Proposition~\ref{naturalporterproj} that for $\sigma' \subseteq \sigma$, the maps $\Omega D(\sigma) \rightarrow \Omega D(\sigma ')$ are given by $\pi \times \Omega p'$ up to homotopy, where $\pi$ is the projection and $p'$ is the pinch map. 
    
    Applying the Hilton--Milnor theorem to the second factor on the right hand side in (\ref{Porteronuxcase}), we obtain a natural homotopy equivalence \[\Omega \left(\bigvee\limits_{\tau \subseteq \sigma,|\tau| \geq 2} \left(\Sigma(\Omega X)^{\wedge \tau}\right)^{\vee|\tau|-1}\right) \simeq \prod\limits_{b \in B_\sigma} \Omega\Sigma\left(\bigwedge_{\tau \subseteq \sigma,|\tau| \geq 2}((\Omega X)^{\tau})^{\wedge b(\tau)}\right).\] By Theorem~\ref{HiltonMilnorproj}, for $\sigma' \subseteq \sigma$, there is a homotopy commutative diagram \[\begin{tikzcd}
	{\Omega \left(\bigvee\limits_{\tau \subseteq \sigma,|\tau| \geq 2} \left(\Sigma(\Omega X)^{\wedge \tau}\right)^{\vee|\tau|-1}\right)} & {\prod\limits_{b \in B_\sigma} \Omega\Sigma\left(\bigwedge_{\tau \subseteq \sigma,|\tau| \geq 2}((\Omega X)^{\tau})^{\wedge b(\tau)}\right)} \\
	{\Omega \left(\bigvee\limits_{\tau' \subseteq \sigma',|\tau'| \geq 2} \left(\Sigma(\Omega X)^{\wedge \tau'}\right)^{\vee|\tau'|-1}\right)} & {\prod\limits_{b \in B_{\sigma'}} \Omega\Sigma\left(\bigwedge_{\tau' \subseteq \sigma',|\tau'| \geq 2}((\Omega X)^{\tau'})^{\wedge b(\tau')}\right)}
	\arrow["{\simeq }", from=1-1, to=1-2]
	\arrow["{\Omega p'}", from=1-1, to=2-1]
	\arrow["\pi'", from=1-2, to=2-2]
	\arrow["\simeq", from=2-1, to=2-2]
\end{tikzcd}\] where $\pi'$ is the projection. Summarising, for $\sigma' \subseteq \sigma$, there is a homotopy commutative diagram \[\begin{tikzcd}
	{\Omega \left(\bigvee\limits_{i \in \sigma} X_i\right)} & {\prod\limits_{i \in \sigma} X_i \times \prod\limits_{b \in B_\sigma} \Omega\Sigma\left(\bigwedge_{\tau \subseteq \sigma,|\tau| \geq 2}((\Omega X)^{\tau})^{\wedge b(\tau)}\right)} \\
	{\Omega \left(\bigvee\limits_{j \in \sigma'} X_j\right)} & {\prod\limits_{j \in \sigma'} X_j \times \prod\limits_{b \in B_{\sigma'}} \Omega\Sigma\left(\bigwedge_{\tau' \subseteq \sigma',|\tau'| \geq 2}((\Omega X)^{\tau'})^{\wedge b(\tau')}\right).}
	\arrow["{\simeq }", from=1-1, to=1-2]
	\arrow["{\Omega p}", from=1-1, to=2-1]
	\arrow["{\pi''}", from=1-2, to=2-2]
	\arrow["{\simeq }", from=2-1, to=2-2]
\end{tikzcd}\]  Hence $\displaystyle \Omega \ux^K_{\mathrm{co}} \simeq \holim_{\catK\op} \Omega D$ is the product of each of the distinct factors that appear in the diagram. For $\sigma' \subseteq \sigma$, the product terms appearing in the decomposition for $\Omega D(\sigma)$ strictly contain the product terms in the decomposition for $\Omega D(\sigma')$. Therefore, enumerating the distinct factors that appear for the maximal faces, we obtain the desired equivalence.
\end{proof}

\begin{example}
    Let $K$ be a $1$-dimensional simplicial complex on $[m]$. In this case, the set $\mathcal{M}$ consists of all the $1$-simplices in $K$. For each $\sigma  = \{i,j\} \in \mathcal{M}$, $B_\sigma = \{\sigma\}$. Therefore, Theorem~\ref{thm:loopsplituxcase} implies there is a homotopy equivalence \[\Omega \ux^K_{co} \simeq \prod\limits_{i=1}^m \Omega X_i \times \prod\limits_{\sigma \in \mathcal{M}} \Omega \Sigma (\Omega X_i \wedge \Omega X_j).\]
\end{example}

\subsection{Loop space decompositions when the domain is contractible}

For a simplicial complex $K$, let $|K|$ be the geometric realisation of $K$ as a topological space. For polyhedral products of the form $\cxx^K$, by \cite[Theorem~2.21]{BBCG2}, there is a homotopy equivalence \begin{equation}\label{eqn:BBCGdecomp}\Sigma \cxx^K \simeq \bigvee\limits_{I \notin K} \Sigma(|K_I| \wedge X^{\wedge I}).\end{equation} In this subsection, we prove a dual statement for polyhedral coproducts of the form $\underline{f}^K_{\mathrm{co}}$ where the domain of each $f_i$ is contractible.  

\begin{theorem}
\label{thm:loopDomainContractible}
    Let $K$ be a simplicial complex on $[m]$ and $f_i\colon 
    X_i \to A_i$ where $X_i$ is contractible and $A_i$ is a pointed, simply connected CW-complex for $1 \leq i \leq m$. Then there is a homotopy equivalence
\begin{equation*}
    \Omega \underline{f}^{K} _{\mathrm{co}} \simeq \prod_{b\in B_{[m]}, I_b \not \in K}\Omega \Map_*(\Sigma |K_{I_b}|, \Sigma \Omega A_{1}^{\wedge l_1(b)} \wedge \cdots \wedge \Omega A_{m}^{\wedge l_m(b)}).
\end{equation*}
\end{theorem}

To prove Theorem~\ref{thm:loopDomainContractible}, we will use the following consequence of Theorem~\ref{thm:loopSplit}.

\begin{lemma}
\label{lem:contractibleFlatProduct}
Assume that $X_i$ is contractible and $A_i$ is a pointed, simply connected CW-complex for all $i$ and $N\in \mathbb{N}^m$. There is a homotopy equivalence 
\begin{equation*}
    \underline{\hat{f}}^{K} _{N,\mathrm{co}} \simeq \Map_*(\Sigma |K|, \Sigma \Omega A_1^{\wedge k_1(N)} \wedge \cdots \wedge \Omega A_m^{\wedge k_m(N)}).
\end{equation*}
\end{lemma}
\begin{proof}
    By definition, $\hat{D}^N(\varnothing)= \Sigma \Omega A_1^{\wedge k_1(N)} \wedge \cdots \wedge \Omega A_m^{\wedge k_m(N)}$ and $\hat{D}^N(\sigma) \simeq *$ for all $\sigma \neq \varnothing$ since all the $X_i$ are contractible.  Thus, we may apply Lemma~\ref{lem:initialDiagramLemma} to the diagram defining $  \underline{\hat{f}}^{K} _{N,\mathrm{co}}$ which yields the claimed result.
\end{proof} 
With the lemma above, it is straightforward to prove Theorem \ref{thm:loopDomainContractible}
\begin{proof}[Proof of Theorem \ref{thm:loopDomainContractible}]
    By Lemma~\ref{lem:contractibleFlatProduct}, if $I_b \in K$, then $\underline{\hat{f}}^{K_{I_b}} _{L_b,\mathrm{co}}$ is contractible. One can then apply Lemma~\ref{lem:contractibleFlatProduct} to the decomposition in Theorem \ref{thm:loopSplit} to prove the statement. 
\end{proof}

\begin{example}
    Let $K = \partial \Delta^{m-1}$. In this case, the only missing face of $K$ is $\{1,\ldots,m\}$. By Theorem~\ref{thm:loopDomainContractible}, there is a homotopy equivalence \begin{equation*}
    \Omega \underline{f}^{K} _{\mathrm{co}} \simeq \prod_{b\in B_{[m]}, I_b = \{1,\ldots,m\}}\Omega \Map_*(\Sigma |K_{I_b}|, \Sigma \Omega A_{1}^{\wedge l_1(b)} \wedge \cdots \wedge \Omega A_{m}^{\wedge l_m(b)}),
\end{equation*} where the indexing set of the product consists of brackets $b$ such that for each $i \in [m]$, there is a face $\sigma \in K$ in $b$ which contains $i$.
\end{example}

In the case of polyhedral products, it is known that the decomposition in~\eqref{eqn:BBCGdecomp} desuspends in certain cases. For example, when $K$ is a shifted complex \cite{GT1,IK1}, a flag complex with chordal 1-skeleton \cite[Theorem~6.4]{PT}, or more generally, a totally fillable simplicial complex \cite[Corollary~7.3]{IK2}. Specialising, polyhedral products of the form $(D^2,S^1)^K$ are known as \textit{moment-angle complexes}, which are denoted $\mathcal{Z}_K$. In the aforementioned cases, $\mathcal{Z}_K$ is homotopy equivalent to a wedge of spheres.  

Consider the case where $K$ is a simplicial complex on $[m]$, and is either a shifted complex, or a flag complex with chordal $1$-skeleton. The dual of the polyhedral product $\cxx^K$ is the polyhedral coproduct $\pxx^K_{\mathrm{co}}$. In the first case, $|K_I|$ is homotopy equivalent to a wedge of spheres for all $I \subseteq [m]$, and in the second case, $|K_I|$ is homotopy equivalent to a set of disjoint points for all $I \subseteq [m]$. Therefore, in the case where each $X_i$ is a simply connected sphere, Theorem~\ref{thm:loopDomainContractible} implies that $\Omega \pxx^K_{\mathrm{co}}$ is homotopy equivalent to a product of iterated loop spaces of spheres. Dual to the polyhedral product case, we give the following conjecture.

\begin{conjecture}
\label{conj:delooping}
    Let $K$ be a shifted complex or a flag complex with chordal $1$-skeleton. Then the decomposition in Theorem~\ref{thm:loopDomainContractible} deloops.
\end{conjecture}

\section{Polyhedral coproducts under operations on simplicial complexes}

\subsection{Joins of simplicial complexes}
For any polyhedral product, if $K = K_1 \star K_2$ is the join of $K_1$ and $K_2$, then $\uxa^K \cong \uxa^{K_1} \times \uxa^{K_2}$. Therefore, we may expect a homotopy equivalence $\uxa^{K}_{\mathrm{co}} \simeq \uxa^{K_1}_{\mathrm{co}} \vee \uxa^{K_2}_{\mathrm{co}}$. However, this does not hold in general for polyhedral coproducts. 

For $1 \leq i \leq 4$, let $X_i = \mathbb{C}P^\infty$, and let $K = \{1,2\} \star \{3,4\}$ be the boundary of a square. Since $(\mathbb{C}P^\infty,*)^{\{1,2\}}_{\mathrm{co}}$ and $(\mathbb{C}P^\infty,*)^{\{3,4\}}_{\mathrm{co}}$ are homotopy equivalent to $\mathbb{C}P^\infty \times \mathbb{C}P^\infty$ by Example~\ref{ex:pointcase}, suppose that $(\mathbb{C}P^\infty,*)_{\mathrm{co}}^K \simeq (\mathbb{C}P^\infty \times \mathbb{C}P^\infty) \vee (\mathbb{C}P^\infty \times \mathbb{C}P^\infty)$. Since $\mathbb{C}P^\infty$ is simply connected and $\Omega \mathbb{C}P^\infty \simeq S^1$, by Theorem~\ref{thm:loopsplituxcase}, there is a homotopy equivalence \[\Omega (\mathbb{C}P^\infty,*)^K_{\mathrm{co}} \simeq \prod\limits_{i=1}^4 (S^1 \times \Omega S^3).\]
	
 Now by Theorem~\ref{naturalporterdecomp} applied to $(\mathbb{C}P^\infty \times \mathbb{C}P^\infty) \vee (\mathbb{C}P^\infty \times \mathbb{C}P^\infty)$, there is a homotopy equivalence \[\Omega((\mathbb{C}P^\infty \times \mathbb{C}P^\infty) \vee (\mathbb{C}P^\infty \times \mathbb{C}P^\infty)) \simeq \prod\limits_{i=1}^4 S^1 \times \Omega \Sigma \left((S^1 \times S^1) \wedge (S^1 \times S^1)\right).\] For spaces $X$ and $Y$, there is a well-known homotopy equivalence $\Sigma(X \times Y) \simeq \Sigma X \vee \Sigma Y \vee \Sigma(X \wedge Y)$. By shifting the suspension coordinate, we obtain homotopy equivalences \[\prod\limits_{i=1}^4 S^1 \times \Omega \Sigma\left((S^1 \vee S^1 \vee S^2) \wedge (S^1 \vee S^1 \vee S^2)\right) \simeq \prod\limits_{i=1}^4 S^1 \times \Omega \Sigma\left(\bigvee\limits_{i=1}^4 S^2 \vee \bigvee\limits_{i=1}^4 S^3 \vee S^4\right).\] By Theorem~\ref{HiltonMilnor}, $\Omega\Sigma \left(\bigvee_{i=1}^4 S^2 \vee \bigvee_{i=1}^4 S^3 \vee S^4\right)$ decomposes as an infinite, finite type product of spheres and loops on spheres. However, since $\Omega (\mathbb{C}P^\infty,*)^K_{\mathrm{co}}$ is homotopy equivalent to a finite product of spheres and loops on spheres, \[\Omega(\mathbb{C}P^\infty,*)^K_{\mathrm{co}} \not\simeq \Omega((\mathbb{C}P^\infty \times \mathbb{C}P^\infty) \vee (\mathbb{C}P^\infty \times \mathbb{C}P^\infty)),\] which implies that \[(\mathbb{C}P^\infty,*)^K_{\mathrm{co}} \not\simeq (\mathbb{C}P^\infty \times \mathbb{C}P^\infty) \vee (\mathbb{C}P^\infty \times \mathbb{C}P^\infty).\]
	
However, it is possible to say something about certain joins. Let $K$ be a simplicial complex on $[m]$ and for $1 \leq i \leq m$, let $(X_i,A_i)$ be CW-pairs. If $(X_{m+1},A_{m+1})$ is a CW-pair where $X_{m+1}$ is contractible, then there are homotopy equivalences $\uxa^{K * \{m+1\}} \cong \uxa^K \times X_{m+1} \simeq \uxa^K$. The following dualises this case.
\begin{proposition}
    \label{prop:joinVertex}
	Let $K$ be a simplicial complex on the vertex set $[m]$ and let $\underline{f}^K_{\mathrm{co}}$ be any polyhedral coproduct. Let $K' = K \star \{m+1\}$ where $f_{m+1} \colon * \to Y$ for some space $Y$. Then $\underline{f}^{K'}_{\mathrm{co}} \simeq \underline{f}^{K}_{\mathrm{co}}$. 
\end{proposition}
\begin{proof}
    There is an equivalence of categories $\cat(K')\op \cong \catK\op \times \cat(\{m+1\})\op$ and a projection map $p \colon \cat(K')\op \to \catK\op$. Geometrically, the map $p$ deletes vertex $m+1$ from any simplex in $K'$. Let $D'\colon \cat(K')\op \to \mathbf{Top}_*$ (resp. $D\colon \cat(K)\op \to \mathbf{Top}_*$) be the diagram defining $\underline{f}^{K'}_{\mathrm{co}}$ (resp. $\underline{f}^{K}_{\mathrm{co}}$) . By right Kan extending $D'$ along $p$, we recover $D$. Note that had we not placed the assumption on the domain of $f_{m+1}$ then the Kan extension would not recover $D$. Since right Kan extensions preserve homotopy limits, we get the equivalence $\underline{f}^{K'}_{\mathrm{co}} \simeq \underline{f}^{K}_{\mathrm{co}}.$
\end{proof}

\subsection{Pullbacks of polyhedral coproducts}

Let $K_1$ be a simplicial complex on $\{1,\ldots,n\}$ and $K_2$ be a simplicial complex on $\{l,\ldots,m\}$ with $n <m$ and $l \leq  m$, and let $L$ be a subcomplex (possibly empty) of $K_1$ and $K_2$ on $\{l,\ldots,n\}$. Define $K = K_1 \cup_L K_2$, and for $M$ one of $K_1$, $K_2$ or $L$, let $\overline{M}$ be the simplicial complex considered on the vertex set $\{1,\ldots,m\}$. For polyhedral products, by \cite[Proposition 3.1]{GT1}, there is a pushout 
\[\begin{tikzcd}
	{\uxa^{\overline{L}}} & {\uxa^{\overline{K_1}}} \\
	{\uxa^{\overline{K_2}}} & {\uxa^{K}.}
	\arrow[from=1-1, to=2-1]
	\arrow[from=1-1, to=1-2]
	\arrow[from=2-1, to=2-2]
	\arrow[from=1-2, to=2-2]
\end{tikzcd}\]
For polyhedral coproducts, we can prove a dual statement. 

\begin{proposition}
    \label{prop:pullbackSquare}
    Let $K_1$ be a simplicial complex on $\{1,\ldots,n\}$ and $K_2$ be a simplicial complex on $\{l,\ldots,m\}$ with $n <m$ and $l \leq  m$, and let $L$ be a subcomplex (possibly empty) of $K_1$ and $K_2$ on $\{l,\ldots,n\}$. Define $K = K_1 \cup_L K_2$. Then there is a homotopy pullback of polyhedral coproducts 
    \[\begin{tikzcd}
	{\underline{f}_{\mathrm{co}}^{K}} & {\underline{f}_{\mathrm{co}}^{\overline{K_2}}} \\
	{\underline{f}_{\mathrm{co}}^{\overline{K_1}}} & {\underline{f}_{\mathrm{co}}^{\overline{L}}}
	\arrow[from=2-1, to=2-2]
	\arrow[from=1-1, to=2-1]
	\arrow[from=1-1, to=1-2]
	\arrow[from=1-2, to=2-2]
\end{tikzcd}\] 
where the maps $\underline{f}_{\mathrm{co}}^{\overline{K_1}} \to \underline{f}_{\mathrm{co}}^{\overline{L}}$, and $\underline{f}_{\mathrm{co}}^{\overline{K_2}} \to \underline{f}_{\mathrm{co}}^{\overline{L}}$ are induced by the simplicial inclusions.
\end{proposition}
\begin{proof}
    Let $\mathcal{C}$ be the category associated to the following diagram
    \begin{equation*}
        \cat(K_1)\op \rightarrow \cat(L)\op \leftarrow \cat(K_2)\op
    \end{equation*}
    where the maps are induced by the inclusions of $L$ into $K_1$ and $K_2$. By Remark \ref{rem:InclusionDiagram}, one may write the elements of the pullback   \[\begin{tikzcd}
	 & {\underline{f}_{\mathrm{co}}^{\overline{K_2}}} \\
	{\underline{f}_{\mathrm{co}}^{\overline{K_1}}} & {\underline{f}_{\mathrm{co}}^{\overline{L}}}
	\arrow[from=2-1, to=2-2]
	\arrow[from=1-2, to=2-2]
\end{tikzcd}\]  as diagrams and we are left with a diagram $D' \colon \mathcal{C} \to \mathbf{Top}_*$.   For each $\sigma \in  L$, let $\sigma_{K_1}$ (resp. $\sigma_{K_2})$ denote the copies in $\mathcal{C}$. The objects of $\mathcal{C}$ are the same as $\catK\op$, but with three copies of each $\sigma \in L$. For each $\sigma \in L$, the maps $D'(\sigma_{K_1}) \to D'(\sigma)$ and $D(\sigma_{K_2}) \to D(\sigma)$ are the identity map. Let $D\colon \catK\op \to \mathbf{Top}_*$ be the diagram defining $\underline{f}_{\mathrm{co}}^{K}$. 
There's a projection map $p\colon \mathcal{C} \to \catK\op$ collapsing the tripled simplices. The right Kan extension of $D'$ along $p$ recovers the diagram $D$. As right Kan extensions preserve homotopy limits, the homotopy pullback diagram has the desired limit. 
\end{proof}

Let $K_1$ and $K_2$ be simplicial complexes and let $K = K_1 \sqcup K_2$. By definition of the polyhedral product, $\ux^K = \ux^{K_1} \vee \ux^{K_2}$. In the case of a polyhedral coproduct $\ux^K_{\mathrm{co}}$, using Proposition ~\ref{prop:pullbackSquare}, we show that the dual holds in this case.

\begin{theorem}
    Let $K_1$ and $K_2$ be simplicial complexes, and let $K = K_1 \sqcup K_2$. There is a homotopy equivalence \[\ux^K_{\mathrm{co}} \simeq \ux^{K_1}_{\mathrm{co}} \times \ux^{K_2}_{\mathrm{co}}.\]
\end{theorem}
\begin{proof}
    By definition, since each $A_i = *$, $\ux^{\varnothing} = *$, and $\ux^{\overline{K_i}} = \ux^{K_i}$ for $i \in \{1,2\}$. Therefore, Proposition ~\ref{prop:pullbackSquare} implies there is a homotopy pullback \[\begin{tikzcd}
	{\ux^{K}_{\mathrm{co}}} & {\ux^{K_2}_{\mathrm{co}}} \\
	{\ux^{K_1}_{\mathrm{co}}} & {*.}
	\arrow[from=1-1, to=1-2]
	\arrow[from=1-1, to=2-1]
	\arrow[from=1-2, to=2-2]
	\arrow[from=2-1, to=2-2]
\end{tikzcd}\] Hence, there is a homotopy equivalence \[\ux^K_{\mathrm{co}} \simeq \ux^{K_1}_{\mathrm{co}} \times \ux^{K_2}_{\mathrm{co}}.\qedhere\]
\end{proof}

\bibliographystyle{amsalpha}

\end{document}